\documentclass[12pt,leqno]{amsart}
\newtheorem{theorem}{Theorem}[section]
\newtheorem{definition}{Definition}[section]

\newtheorem{corollary}{Corollary}[section]

\newtheorem{remark}{Remark}[section]

\newtheorem{example}{Example}[section]

\newcommand{\be}{\begin{equation}}
\newcommand{\ee}{\end{equation}}
\numberwithin{equation}{section}
\newcommand{\bea}{\begin{eqnarray}}
\newcommand{\eea}{\end{eqnarray}}
\newcommand{\beb}{\begin{eqnarray*}}
\newcommand{\eeb}{\end{eqnarray*}}
\usepackage{amssymb,amsfonts,amsthm,setspace,indentfirst}
\usepackage[dvips]{graphics}
\usepackage{epsfig}
\begin{document}
\title{Rough statistical convergence of sequences in a partial metric space}
\author{Sukila khatun$^{1}$ and Amar Kumar Banerjee$^{2}$}
\address{$^{1}$Department of Mathematics, The University of Burdwan,
Golapbag, Burdwan-713104, West Bengal, India.} 
\address{$^{2}$Department of Mathematics, The University of Burdwan,
Golapbag, Burdwan-713104, West Bengal, India.}
\email{$^{1}$sukila610@gmail.com}
\email{$^{2}$akbanerjee@math.buruniv.ac.in, akbanerjee1971@gmail.com}

\begin{abstract}
In this paper, using the concept of natural density, we have introduced the notion of rough statistical convergence which is an extension of the notion of rough convergence in a partial metric space. We have defined the set of rough statistical limit points of a sequence in a partial metric space and proved that this set is closed and bounded. Finally, we have found out the relationship between the set of statistical cluster points and the set of rough statistical limit points of  sequences in a partial metric space.
\end{abstract}
\subjclass[2020]{40A05, 40A99.}
\keywords{Natural density, statistical convergence, rough convergence, partial metric spaces, rough limit sets.}
\maketitle
\section{\bf{Introduction and Preliminaries}}
The idea of rough convergence of sequences was first introduced in a normed linear space by H. X. Phu \cite{PHU} in 2001. The formal definition given by him is as follows: if $\{x_n\}$ is a sequence in a normed linear space $(X,||.||)$ and $r$ is a non-negative real number, then $\{x_n\}$ is said to be $r$-convergent to $x$, denoted by $x_{n} \stackrel{r} \longrightarrow x$, if $ \forall \varepsilon>0 $, $\exists \ i_\varepsilon \in \mathbb N : i \geq i_\varepsilon \longrightarrow ||x_n-x||<r+\varepsilon$, or equivalently, if $limsup_{n\to\infty} \ ||x_n-x||<r$. This is the idea of rough convergence with $r$ as roughness degree. For $r=0$ we have the ordinary convergence. The main difference of ordinary and rough convergence is that if a sequence in a normed linear space converges, then it's limit is unique but for rough convergence with roughness degree $r>0$ the limit may be infinite. The set of all $r$-limits is denoted by $LIM^{r}x_{n}$. Phu discussed about the idea of rough limit sets and also some basic properties of rough limit sets such as boundedness, closedness and convexity etc. Also he introduced the idea of rough Cauchy sequences. Later works on rough convergence in many directions were carried out by many authors  \cite{{AYTER1}, {RMROUGH}, {DR}, {NH}, {PMROUGH1}, {PMROUGH2}}.
We intended to study the idea of rough convergence in a more generalized form. So, it is required to discuss in brief the ideas of statistical convergence.
Statistical convergence is a generalization of the ordinary convergence. The concept of Statistical convergence was introduced by H. Fast \cite{HF} and H. Steinhaus \cite{HS} in the year 1951. The idea of statistical convergence has been depend on the structure of ideals of subsets of the natural numbers by T. Salat \cite{TS} as follows: if $ B \subset \mathbb{N}$, then $B_{n}$ will denote the set $ \{k \in B : k \leq n \}$ and $|B_{n}|$ stands for the cardinality of $B_{n}$. The natural density of $B$ is denoted by $d(B)$ and defined by $d(B)=lim_{n\to\infty} \frac{|B_{n}|}{n}$, if the limit exists. A real sequence $\{ \xi_n \}$ is said to be statistically convergent to $\xi$ if for every $\varepsilon>0$ the set $B(\varepsilon)= \{ k \in \mathbb{N}: |\xi_{n} - \xi| \geq \varepsilon \}$ has natural density zero. In this case, $\xi$ is called the statistical limit of the sequence $\{ \xi_n \}$ and we write $st-LIM^{r} \xi_{n}= \xi$. We have $d(B^c)=1-d(B)$, where $B^c=\mathbb{N} \setminus B$ is the complement of $B$. If $B_1 \subset B_2$, then $d(B_1) \leq d(B_2)$.
In 1994, partial metric spaces were introduced by S. Mattews \cite{MATW} as a generalization of metric spaces. If $(X,p)$ is a partial metric space, then the concept of self-distance $p(x,x)$ which is not necessarily zero for all $x \in X$. In our present work we discuss the idea of rough statistical convergence of sequences in partial metric spaces. We have given the idea of statistical boundedness in a partial metric space. We have also defined the set of rough statistical limit points and found out several properties of this set like boundedness and closedness etc. 

\begin{definition}\cite{BMW}
A partial metric on a nonempty set $X$ is a function $p: X\times X \longrightarrow [0, \infty)$ such that for all $x,y,z \in X$:\\
$(p1)$ $0 \leq p(x,x) \leq p(x,y)$ (nonnegativity and small self-distances),\\
$(p2)$ $x=y \Longleftrightarrow p(x,x)=p(x,y)=p(y,y)$ (indistancy both implies equality),\\
$(p3)$  $p(x,y)= p(y,x)$ (symmetry),\\
$(p4)$  $p(x,y) \leq p(x,z) + p(z,y) - p(z,z)$ (triangularity).\\
Then $(X,p)$ is said to be a partial metric space, where $X$ is a nonempty set and $p$ is a partial metric on $X$.

Properties and examples of partial metric spaces have been thoroughly discussed in \cite{BMW}.
\end{definition}

\begin{definition} \cite{BMW}
In a partial metric space  $(X,p)$, for $r>0$ and $x \in X$ we define the open and closed ball of radius $r$ and center $x$ respectively as follows  :
\begin{center}
    $B^{p}_{r}(x)=\{ y \in X : p(x,y)<p(x,x)+r  \}$ \\
$\overline{B^{p}_{r}}(x)=\{ y \in X : p(x,y) \leq p(x,x)+r \}.$ 
\end{center}
\end{definition}

\begin{definition} \cite{BMW}
Let $(X,p)$ be a partial metric space. A subset $U$ of $X$ is said to be a bounded in $X$ if there exists a positive real number $M$ such that $sup$ $\{ p(x,y): x,y \in U\}< M$.
\end{definition}

\begin{definition} \cite{BMW}
Let $(X,p)$ be a partial metric space and $\{x_{n}\}$ be a sequence in $X$. Then $\{x_{n}\}$ is said to converge to $x \in X$ if and only if $lim_{n\to\infty}p(x_{n},x)=p(x,x)$; i.e. if for each $\epsilon > 0$ there exists $k \in \mathbb{N}$ such that 
$|p(x_{n},x)-p(x,x)|< \epsilon$ for all $n \geq k$.
\end{definition}

\begin{definition} \cite{FN}
    Let $(X,p)$ be a partial metric space. Then the sequence $\{x_n\}$ is said to be statistically convergent to $x$ if for every $\varepsilon>0$, $d({n \in \mathbb N: |p(x_{n},x)-p(x,x)| \geq \varepsilon})=0$.
\end{definition}

\begin{definition} \cite{FN}
    Let $(X,p)$ be a partial metric space. Then the sequence $\{x_n\}$ is called statistically Cauchy if for every $\varepsilon>0$ there is a positive integer $m$ and $l \geq 0$ such that 
    \begin{center}
        $d({n \in \mathbb N: |p(x_{n},x_{m})-l| \geq \varepsilon})=0$.
    \end{center}
\end{definition}

\section{\bf{Rough statistical convergence in partial metric spaces}}

\begin{definition}
    A sequence $\{x_n\}$ in partial metric spaces $(X,p)$ is said to be rough statistical convergent (or in short r-statistical convergent or r-st convergent) to $x$ of roughness degree $r$ for some non-negative real number  $r$ if for every $\varepsilon>0$, $d(\{n \in \mathbb N : |p(x_n,x)-p(x,x) | \geq r+\varepsilon\})=0$.
\end{definition}
We denote it by the notation $x_{n} \stackrel{r-st}{\longrightarrow} x$ in $(X,p)$. When $r=0$, then the rough statistical convergent becomes the statistical convergent in partial metric spaces $(X,p)$. Let a sequence $\{x_n\}$ is rough statistical convergent to $x$, then $x$ is said to be a rough statistical limit point of $\{x_n\}$. And the set of all rough statistical limit points of a sequence $\{x_n\}$ is said to be the $r-st \ limit \ set$. We will denote it by $st-LIM^r x_n$. Hence $st-LIM^{r}x_{n}= \left\{x \in X : x_{n} \stackrel{r-st}{\longrightarrow} x \right\}$. 
\begin{theorem}
    Every rough convergent sequence in a partial metric space $(X,p)$ is rough statistically convergent in $(X,p)$.
\end{theorem}
\begin{proof}
    Let $\{x_n\}$ be a sequence in $(X,p)$ and rough convergent to $x$.
    Then for an arbitrary $\varepsilon>0$, there exists $k \in \mathbb N$ such that $ |p(x_n,x)-p(x,x)| < r+\varepsilon$, $\forall n \geq k$.
    Now, the set $A=\{n \in \mathbb N : |p(x_n,x)-p(x,x) | \geq r+\varepsilon\} \subset \{1,2,3,......,(k-1) \}$ is a finite set.
    So, $d(A)=0$.
    Hence $\{x_n\}$ is rough statistically convergent in $(X,p)$.
    
\end{proof}

\begin{remark}
    The converse of the above theorem (2.1) may not be true i.e. rough statistically convergent sequence may not be rough convergent in $(X,p)$.
\end{remark}

\begin{example}
  Let $X=\mathbb R^+$ and $p: X\times X \longrightarrow \mathbb R^+$ be given by $p(x,y)=max \{x,y \}$ for all $x,y \in X$. Then $(X,p)$ is a partial metric space.
  Let us took a sequence $\{x_{n}\}$ which is defined by 
  \begin{equation*}
    \  x_{n}= \begin{cases}
          k, & \text { if $n=k^2$ }, \\
          2, & \text { if $n \neq k^2$ and $n$ is even}, \\
          0, & \text { if $n \neq k^2$ and $n$ is odd. }
                 \end{cases}
  \end{equation*}
 Let $\varepsilon>0$ be given. \\
 Then $A_1=\{n \in \mathbb N : |p(x_n,0)-p(0,0) | \geq 1+\varepsilon\} \subset \{1^2,2^2,3^2,4^2,5^2.....\}=P$ (say).\\
 Similarly, the sets
  $A_2=\{n \in \mathbb N : |p(x_n,1)-p(1,1) | \geq 1+\varepsilon\} \subset P$
 and $A_3=\{n \in \mathbb N : |p(x_n,2)-p(2,2) | \geq 1+\varepsilon\} \subset P$ .\\
 Since $d(P)=0$, so $d(A_1)=d(A_2)=d(A_3)=0$.\\
 So, $\{x_{n}\}$ is rough statistical convergent to 0,1 and 2 of roughness degree 1.\\
 Again, for any $k$ such that $1<k<2$, the set $A_k=\{n \in \mathbb N : |p(x_n,k)-p(k,k) | \geq 1+\varepsilon\} \subset \{1^2,2^2,3^2,.....\}=P$.\\
 If $K>2$, then the set $A_k \subset P$.
 So, the statistical rough limit set of roughness degree 1 is $ \{0 \} \cup \{k \in \mathbb R^+ : k \geq 1 \}$ i.e. $st-LIM^r x_n= \{0 \} \cup [1,\infty)$.
 But \begin{equation*}
     |p(x_n,2)-p(2,2)|= \begin{cases}
         |p(k,2)-p(2,2)|=|k-2|, & \text{if $n=k^2$},\\
         |p(2,2)-p(2,2)|=0, & \text{if $n \neq k^2$ and $n$ is even}, \\
         |p(0,2)-p(2,2)|=0, & \text{if $n \neq k^2$ and $n$ is odd}.
     \end{cases}
 \end{equation*}
 So, when $n=k^2$, there dose not exist any positive integer $n_0$ such that the condition $|p(x_n,2)-p(2,2)|< r+ \varepsilon$ for all $n \geq n_0$ holds, since $|k-2| \longrightarrow \infty$ as $k^2 \longrightarrow \infty$.\\
 Hence $\{x_{n}\}$ is not rough convergent to 2 of any roughness degree $r>0$.
 Similarly, it can be shown that  $\{x_{n}\}$ is not rough convergent 0 or any number in $[1, \infty)$ i.e. $LIM^r x_n= \phi$.
\end{example}

\begin{definition} \cite{SUK2}
    The diameter of a set $B$ in a partial metric space $(X,p)$ is defined by 
    \begin{center}
    $diam(B)$ = $sup$ $\{ p(x,y) : x, y\in B\}$.
\end{center}
\end{definition}

\begin{theorem}
Let $(X,p)$ be a partial metric space and $a$ be a positive real number such that $p(x,x)=a$ for all $x$ in $X$. Then for a sequence $\{x_{n}\}$, we have $diam(st-LIM^rx_n) \leq (2r+2a)$. 
\end{theorem}

\begin{proof}
Let $diam(st-LIM^{r} x_{n}) > 2r+2a$.
Then there exist elements $y,z \in st-LIM^{r} x_{n}$ such that $p(y,z)>2r+2a$.
Let us take $\varepsilon \in (0,\frac{p(y,z)}{2}-r-a)$. 
Since $y,z \in st-LIM^{r} x_{n}$, we have $d(M_1)=0$ and $d(M_2)=0$,  where 
$ M_1= \{ n \in \mathbb N: |p(x_{n},y)-p(y,y)| \geq r+\varepsilon \}$ and
$ M_2= \{ n \in \mathbb N:|p(x_{n},z)-p(z,z)|\geq r+\varepsilon \}$.
Using the property of natural density, we get $d(M^{c}_1 \cap M^{c}_2)=1$.
So, for all $n \in M^{c}_1 \cap M^{c}_2 $, we have 
 \begin{equation*}
       \begin{split}
p(y,z) &\leq p(y,x_{n})+p(x_{n},z)-p(x_{n},x_{n})\\
       &=\{p(x_{n},y)-p(y,y)\}+\{p(x_{n},z)-p(z,z)\}-p(x_{n},x_{n})+p(y,y)+p(z,z)\\
       &< 2(r+\varepsilon)-a+a+a\\
       &=2r+2\varepsilon+a\\
       &<2r+ p(y,z)-2r-2a+a\\
       &=p(y,z)-a , \ \text{which is a contradiction}.    
       \end{split}
   \end{equation*}    
Hence we must have $diam(st-LIM^{r} x_{n}) \leq 2r+2a$. 
\end{proof}

\begin{theorem}
    If a sequence $\{x_{n}\}$ statistically converges to $x$ in a partial metric space $(X,p)$, then $\{ y \in \overline{B^{p}_{r}}(x):p(x,x)=p(y,y)\} \subseteq st-LIM^{r}x_{n}$.
\end{theorem}

\begin{proof}
 Let $\varepsilon >0$ and a sequence $\{x_{n}\}$ statistically converges to $x$ in a partial metric space $(X,p)$.
 Then $d(A)=0$, where $A=\{n \in \mathbb N: |p(x_{n},x)-p(x,x)| \geq \varepsilon \}$.
 Let $ y \in \overline{B^{p}_{r}}(x)$ such that $p(x,x)=p(y,y) $. Then $p(x,y) \leq p(x,x)+r$ such that $p(x,x)=p(y,y)$........(1). \\
 For $n \in A^c$,
 \begin{equation*}
    \begin{split}
        p(x_{n},y) &\leq p(x_{n},x)+p(x,y)-p(x,x)\\
             &\leq \{p(x_{n},x)-p(x,x)\}+p(x,y)\\
             &< \epsilon + \{p(x,x)+r\} \ \text{by} \ (1)  \\
             &= p(x,x)+(r+\epsilon) 
    \end{split}
\end{equation*}
 Therefore, 
\begin{equation*}
    \begin{split}
p(x_{n},y)-p(y,y) & <p(x,x)-p(y,y)+(r+\epsilon)\\
                  & =(r+\epsilon), \ \ \text{since} \ p(x,x)= p(y,y).
    \end{split}
\end{equation*}
So, by (p1) axiom, $|p(x_{n},y)-p(y,y)|=p(x_{n},y)-p(y,y) < (r+\epsilon)$, for every $n\in A^c$. So, $\{ n \in \mathbb N: |p(x_{n},y)-p(y,y)| \geq r+\varepsilon \} \subset A$ and hence $d( \{ n \in \mathbb N : |p(x_{n},y)-p(y,y)| \geq r+\varepsilon \} )=0$. So $y \in st-LIM^{r}x_{n}$.
\end{proof}

\begin{theorem}
   Let $\{x_{n}\}$ be a $r$-statistical convergent sequence in $(X,p)$ and $\{y_{n}\}$ be a convergent sequence in $st-LIM^{r}x_{n}$ converging to $y$. Then $y$ must belongs to $st-LIM^{r}x_{n}$.
\end{theorem}
\begin{proof}
If $st-LIM^{r}x_{n}=\phi$, then there is nothing to prove.
So, we can assume that $st-LIM^{r}x_{n} \neq \phi$.
Let $\{y_{n}\}$ be a sequence in $st-LIM^{r}x_{n}$ such that $\{y_{n}\} {\longrightarrow} y$.
Let $\varepsilon>0$ be given. Since $\{y_{n}\} {\longrightarrow} y$, there exists $n_\frac{\varepsilon}{2} \in \mathbb N$ such that $|p(y_n,y)-p(y,y)| < \frac{\varepsilon}{2}$ for all $n \geq n_\frac{\varepsilon}{2}$.
Now, choose an $n_0 \in \mathbb N$ such that $n_0 > n_\frac{\varepsilon}{2}$.
Then we can write $|p(y_{n_0},y)-p(y,y)| < \frac{\varepsilon}{2}$.
On the other hand, because $\{y_{n}\} \subset st-LIM^{r}x_{n}$, we have $y_{n_0} \in st-LIM^{r}x_{n} $ and $d(\{n \in \mathbb N: |p(x_n,y_{n_0})-p(y_{n_0},y_{n_0})| \geq r+ \frac{\varepsilon}{2}\})=0$..........(1)\\
Now, we show \\
$ \{n \in \mathbb N :|p(x_n,y)-p(y,y)|<r+\varepsilon \} \supseteq \{ n \in \mathbb N: |p(x_n,y_{n_0})-p(y_{n_0},y_{n_0})| < r+ \frac{\varepsilon}{2} \}$.......(2). \\
Let $ k \in \{ n \in \mathbb N: |p(x_n,y_{n_0})-p(y_{n_0},y_{n_0})| < r+ \frac{\varepsilon}{2} \} $. \\
Then we have $|p(x_k,y_{n_0})-p(y_{n_0},y_{n_0})| < r+ \frac{\varepsilon}{2}$ and hence 
\begin{equation*}
    \begin{split}
         p(x_{k},y)-p(y,y) & \leq p(x_{k},y_{n_0})+p(y_{n_0},y)-p(y_{n_0},y_{n_0})-p(y,y) \\
         & \leq | p(x_{k},y_{n_0})-p(y_{n_0},y_{n_0})+p(y_{n_0},y)-p(y,y)| \\
         & \leq |p(x_{k},y_{n_0})-p(y_{n_0},y_{n_0})| + |p(y_{n_0},y)-p(y,y)| \\
         & < (r+\frac{\varepsilon}{2}) + \frac{\varepsilon}{2} \\
         & = r+\varepsilon
    \end{split}
\end{equation*}
So, by (p1) axiom, $|p(x_{k},y)-p(y,y)|<r+\varepsilon$. Therefore, $ k \in \{ n \in \mathbb N: |p(x_{n},y)-p(y,y)|<r+\varepsilon \} $, which proves (2).
From(1), we can say that the set on the right-hand side of (2) has natural density 1 and so, the natural density of the set on the left-hand side of (2) is equal to 1. Hence 
$d(\{n \in \mathbb N :|p(x_n,y)-p(y,y)| \geq r+\varepsilon \})=0$.
This proves that $y \in st-LIM^{r}x_{n}$. 
\end{proof} 

\begin{corollary}
Let $\{x_{n}\}$ be a $r$-statistical convergent sequence in a partial metric space $(X,p)$. Then $st-LIM^{r}x_{n}$ is a closed set for any degree of roughness $r \geq 0$.
\end{corollary}
\begin{proof}
 Since the partial metric space $(X,p)$ is first countable \cite{SUK2}, the result follows directly from theorem (2.4).   
\end{proof}

\begin{definition} (cf \cite{AKB})
   A sequence $\{x_{n}\}$ in a partial metric space $(X,p)$ is said to be statistically bounded if for any fixed $u \in X$ there exists a positive real number $M$ such that 
   \begin{center}
       $d(\{ n \in \mathbb N : p(x_n,u) \geq M \})=0$.
   \end{center} 
\end{definition}

\begin{remark}
    In \cite{FN}, the definition of boundedness of a sequence in a partial metric space is given as follows: \\
    $\{x_{n}\}$  is bounded if there exist $M>0$ such that $p(x_n,x_m) \leq M$, $\forall \ n,m \in \mathbb N $. \\
    This definition is equivalent to the definition (2.3).
    
\end{remark}

\begin{theorem}
Let $(X,p)$ be a partial metric space and $a$ be a positive real number such that $p(x,x)=a$, $ \forall x \in X$. Then a sequence $\{x_{n}\}$ is statistically bounded in $(X,p)$ if and only if there exists a non-negative real number $r$ such that $st-LIM^{r}x_{n} \neq \phi$.     
\end{theorem}

\begin{proof}
Let $u \in X$ be a fixed element in $X$.
Since the sequence $\{x_{n}\}$ is statistically bounded, there exists a positive real number $M$ such that $d(\{ n \in \mathbb N : p(x_n,u) \geq M \})=0$.
Let $\varepsilon>0$ be arbitrary and $r=M+a$.
Now, we show that
\begin{center}
    $\{ n \in \mathbb N : |p(x_n,u)-p(u,u) < r+\varepsilon \} \supset \{ n \in \mathbb N : p(x_n,u) < M \}$ ..........(1).
\end{center}
Let $i \in \{ n \in \mathbb N : p(x_n,u) < M \}$.
Then $p(x_i,u) < M$.
Now, 
$p(x_i,u)-p(u,u) \leq |p(x_i,u)-p(u,u)| \leq |p(x_i,u)|+|p(u,u)| < M+a=r <r+\varepsilon $.
So, $ i \in \{ n \in \mathbb N : |p(x_n,u)-p(u,u)| < r+\varepsilon \} $. So, (1) holds and hence 
$\{ n \in \mathbb N : |p(x_n,u)-p(u,u) \geq r+\varepsilon \} \subset \{ n \in \mathbb N : p(x_n,u) \geq M \}$. This implies that $d(\{ n \in \mathbb N : |p(x_n,u)-p(u,u) \geq r+\varepsilon \})=0$. Therefor $ u \in st-LIM^{r}x_{n} $ i.e.
$st-LIM^{r}x_{n} \neq \phi$. \\

Conversely, suppose that $st-LIM^{r}x_{n} \neq \phi$.
So, let $u$ be a $r$-limit of $\{x_{n}\}$. Therefore, for $\varepsilon >0$, 
$d(\{ n \in \mathbb N : |p(x_n,u)-p(u,u) \geq r+\varepsilon \})=0$.
Let $K=\{ n \in \mathbb N : |p(x_n,u)-p(u,u) \geq r+\varepsilon \}$ and $ M=r+a+2\varepsilon$. Then $d(K)=0$ and if $ n \in K^c$, then 
\begin{equation*}
    \begin{split}
     p(x_n,u) & = |p(x_n,u)-p(u,u)+p(u,u)|  \\
              & \leq |p(x_n,u)-p(u,u)|+|p(u,u)| \\
              & < r+\varepsilon+a \\
              & < M
    \end{split}
\end{equation*}
So, $ n \in \{ n \in \mathbb N : p(x_n,u) < M \} $.
This implies $ K^c \subset \{ n \in \mathbb N : p(x_n,u) < M \}$. \\
So, $ \{ n \in \mathbb N : p(x_n,u) \geq M \} \subset K $.
Since $d(K)=0$, $d(\{ n \in \mathbb N : p(x_n,u) \geq M \})=0$. \\
Hence $\{x_{n}\}$ is statistically bounded. 
\end{proof}

\begin{theorem}
    Let $\{ x_{{n}_{k}}\}$ be a subsequence of $\{x_{n}\}$ such that $d(\{ n_1, n_2,....... \})=1$, then $st-LIM^{r}x_{n} \subseteq st-LIM^{r}x_{n_{k}}$. 
\end{theorem}

\begin{proof}
Let  $\{ x_{{n}_{k}}\}$ be a subsequence of $\{x_{n}\}$ and $ x \in st-LIM^{r}x_{n}$ and let $ \varepsilon >0$. 
So the set $A=\{ n \in \mathbb N : |p(x_n,x)-p(x,x)| \geq r+\varepsilon \}$ has density zero. 
So, $d(A^c)=1$. 
Since the set $K= \{ n_1, n_2,....... \}$ has density 1, $A^c \cap K \neq \phi$.
For if $A^c \cap K= \phi$, then $K \subset A$ and so $d(K)=0$, since $d(A)=0$. But $d(K)=1$. Therefore, $A^c \cap K \neq \phi$.
Let $ n_q \in A^c \cap K $. Then $ |p(x_{{n}_{q}},x)-p(x,x)| < r+\varepsilon $ i.e. $ n_q \in \{ n_k \in K: |p(x_{{n}_{k}},x)-p(x,x)| < r+\varepsilon  \}$. 
So, $ \{ n_k \in K: |p(x_{{n}_{k}},x)-p(x,x)| \geq r+\varepsilon  \} \subset A \cup K^c$. This implies that $ d( \{ n_k \in K: |p(x_{{n}_{k}},x)-p(x,x)| \geq r+\varepsilon  \} )=0$, since $d(A \cup K^c) \leq d(A)+d(K^c)=0+0=0$. 
Therefore $ x \in st-LIM^{r}x_{n_{k}}$. Hence $st-LIM^{r}x_{n} \subseteq st-LIM^{r}x_{n_{k}}$.
\end{proof}

\begin{theorem}
Let $\{x_{n}\}$ and $\{y_{n}\}$ be two sequences in $(X, p)$ such that $p(x_{n}, y_{n}) \longrightarrow 0$ as $n \longrightarrow \infty$.
 If $\{x_{n}\}$ is $r$-statistical convergent to $x$ and $p(x_{n}, x_{n}) \longrightarrow 0$ as $n \longrightarrow \infty$, then $\{y_{n}\}$ is $r$-statistical convergent to $x$. 
 Conversely, if $\{y_{n}\}$ is $r$-statistical convergent to $y$ and $p(y_{n},y_{n})\longrightarrow 0 $ as $ n \longrightarrow \infty $, then $\{x_{n}\}$ is $r$-statistical convergent to $y$.  
\end{theorem}

\begin{proof}
 Let $\{x_{n}\}$ be $r$-statistical convergent to $x$ and let $\varepsilon>0$ be arbitrary. 
 So, for $\varepsilon>0$, $d(A)=0$, where $A=\{n \in \mathbb N: |p(x_n,x)-p(x,x)| \geq r+\frac{\varepsilon}{3} \}$. \\
 Again, since $p(x_{n}, y_{n}) \longrightarrow 0$ as $n \longrightarrow \infty$, for $\varepsilon>0$, $ \exists \ k_1 \in \mathbb N$ such that $p(x_{n}, y_{n})\leq \frac{\varepsilon}{3}$, when $n \geq k_1$ ........(1) \\
Since $p(x_{n}, x_{n}) \longrightarrow 0$ as $n \longrightarrow \infty$, $ \exists \ k_2 \in \mathbb N$ such that $p(x_{n}, x_{n})\leq \frac{\varepsilon}{3}$, when $n \geq k_2$...........(2) \\
Let $ K= \ max \{ k_1, k_2 \}$. Then for $n \geq k$, (1) and (2) both hold. \\
We can write $p(y_{n},x) \leq p(y_{n},x_{n})+p(x_{n},x)-p(x_{n},x_{n})$.
So, $p(y_{n},x) - p(x,x) \leq p(y_{n},x_{n}) + p(x_{n},x) - p(x_{n},x_{n}) - p(x,x)$.\\
Since $d(A)=0$, $d(A^c)=1$.
Since $ d( \{1,2,......k \})=0$, so $ d( \{1,2,......k \}^c)=1$. 
So, $A^c \cap \{1,2,......k \}^c \neq \phi $.
Therefor, if $ n \in A^c \cap \{1,2,......k \}^c $, then
\begin{equation*}
    \begin{split}
|p(y_{n},x) - p(x,x)|& = |p(y_{n},x) - p(x,x)| \\
    & \leq |p(y_{n},x_{n}) + p(x_{n},x) - p(x_{n},x_{n}) - p(x,x)| \\
    &\leq |p(x_{n},y_{n})| + |p(x_{n},x) - p(x,x)| + |p(x_{n},x_{n})|\\
    & < \frac{\varepsilon}{3} + (r+ \frac{\varepsilon}{3})+\frac{\varepsilon}{3} \\
    &  = r + \varepsilon 
    \end{split}
\end{equation*}
This implies that $ A^c \cap \{1,2,......k \}^c \subset \{n \in \mathbb N : |p(y_{n},x) - p(x,x)| < r+\varepsilon \} $ and hence $ \{n \in \mathbb N : |p(y_{n},x) - p(x,x)| \geq r+\varepsilon \} \subset (A^c \cap \{1,2,......k \}^c)^c = A \cup \{1,2,......k \} $.
Since $ d(A \cup \{1,2,......k \}) \leq d(A)+d(\{1,2,......k \}) = 0+0=0$, it follows that 
 $d(\{n \in \mathbb N: |p(y_{n},x) - p(x,x)| \geq r+\varepsilon \})=0$.
Therefore, $\{y_{n}\}$ is $r$-statistical convergent to $x$.

Converse part is similar.
\end{proof}

\begin{theorem}
Let $\{x_{n}\}$ and $\{y_{n}\}$ be two sequences in $(X, p)$ such that $p(x_{n}, y_{n}) \longrightarrow 0$ as $ n \longrightarrow \infty$. 
If $\{x_{n}\}$ is $r$-statistical convergent to $x$ and a positive number $c$ such that $ p(x_{n}, x_{n}) \leq c$ for all $n$ (i.e. self-distance of the sequence is less or equal to $c$), then  $\{y_{n}\}$ is $(r + c) $-statistical convergent to $x$. 
Conversely, if $\{y_{n}\}$ is $r$-statistical convergent to $y$ and a positive number $d$ such that $p(y_{n}, y_{n})\leq d$ for all $n$, then $\{x_{n}\}$ is $(r + d)$-statistical convergent to $y$.      
\end{theorem}

\begin{proof}
The proof is parallel to the proof of the above theorem and so is omitted. 
\end{proof}

\begin{definition}
Let $(X, p)$ be a partial metric space. Then $c \in X$ is called a statistical cluster point of a sequence $\{x_{n}\}$ in $(X, p)$ if for every $\varepsilon>0$, $d(\{n \in \mathbb N: |p(x_n,c)-p(c,c)| < \varepsilon \}) \neq 0$.
\end{definition}

\begin{theorem}
 Let $(X, p)$ be partial metric space and $a$ be a real constant such that $p(x,x)=a$, $\forall x \in X$. Let $\{x_{n}\}$ be a sequence in $(X, p)$. If $c$ is a cluster point of $\{x_{n}\}$, then $st-LIM^{r}x_{n}  \subset \overline{B^{p}_{r}}(c)$ for some $r >0$.
\end{theorem}

\begin{proof}
If possible suppose that $y \in st-LIM^{r}x_{n}$ but $y \notin \overline{B^{p}_{r}}(c)=\{y \in X: p(c,y) \leq p(c,c) + r\} =\{y \in X: p(c,y) \leq a + r\}$.
So, $a + r < p(c,y)$.
Let $\varepsilon^{'}= p(c,y) - (a+r)$, so that $p(c,y)= \varepsilon^{'} + a + r$, where $\varepsilon^{'} > 0$.
Choose $\varepsilon =\frac{\varepsilon^{'}}{2}$ and so we can write 
$p(c,y)= 2\varepsilon + a + r$.
Then $B_{r+\varepsilon}(y) \cap B_{\varepsilon}(c) = \phi$.
[ For, if $ q \in B_{r+\varepsilon}(y) \cap B_{\varepsilon}(c)$,
then it would imply that $p(q,y) < p(y,y) + r + \varepsilon = a + r + \varepsilon$  and $p(q,c) < p(c,c) + \varepsilon = a + \varepsilon$ which in turn implies that
 $p(c,y) \leq p(c,q) + p(q,y) - p(q,q) < \{a+\varepsilon \} + \{a+r+\varepsilon \} -a = a+r+2\varepsilon = p(c,y)$, a contradiction ].
Therefore, $B_{r+\varepsilon}(y) \cap B_{\varepsilon}(c) = \phi$.
But since $y \in st-LIM^{r}x_{n}$, for $\varepsilon > 0$, $d(A_1)=0$, where $A_1=\{n \in \mathbb{N}: |p(x_{n},y)-p(y,y)| \geq r+\varepsilon \}$.
Again, since $c$ is a cluster point of $\{x_{n}\}$, for the same $\varepsilon > 0$, $d(A_2) \neq 0$, where $A_2=\{n \in \mathbb{N}: |p(x_{n},c)-p(c,c)| < \varepsilon \}$.
Now, let $k \in A_{1}^c \cap A_{2}$, then 
 $|p(x_{k},c)-p(c,c)| < \varepsilon$. 
 This implies that $p(x_{k},c) - p(c,c) < \varepsilon$. So, $p(c,x_{k}) < p(c,c) + \varepsilon$. Hence $ x_{k} \in B_{\varepsilon}(c)$.
Also, $|p(x_{k},y)-p(y,y)| < r+\varepsilon$.
So, $p(x_{k},y)-p(y,y) < r+\varepsilon$. This implies that $p(y,x_{k}) < p(y,y) + r + \epsilon$. Hence $ x_{k} \in B_{r+\varepsilon}(y)$.
So, $ x_{k} \in B_{r+\varepsilon}(y) \cap B_{\varepsilon}(c)$, which is a contradiction.
Hence  $y \in \overline{B^{p}_{r}}(c)$.
\end{proof}

\subsection*{Acknowledgements}
The first author is thankful to The University of Burdwan for the grant of Senior Research Fellowship (State Funded) during the preparation of this paper. Both authors are also thankful to DST, Govt of India for providing FIST project to the dept. of Mathematics, B.U. \\

\end{document}